\documentclass[a4paper,12pt,english]{article}
\usepackage[cp1251]{inputenc}
\usepackage[T2A]{fontenc}
\usepackage[english]{babel}
\usepackage[dvips]{graphicx}

\usepackage{amsmath}
\usepackage{amssymb}
\usepackage{amsxtra}
\usepackage{amsthm,amscd}
\usepackage{tikz-cd}

\usepackage{geometry}
\def\be{e}
\def\bi{n}
\def\oq{\omega}
\def\aq{\alpha}
\def\bq{\beta}

\def\tq{\theta}
\def\ro{\rho}
\def\lqq{\lambda}
\def\rmd{{\rm d}}
\def\rmi{{\rm i}}
\def\vk{\varkappa}
\def\sg{\sigma}
\def\bq{\mathbf{q}}
\def\mB{\mathfrak{B}}
\newcommand{\bb}[1]{\langle #1 \rangle}
\newcommand{\tl}[1]{{\tilde #1}}
\newcommand{\ov}[1]{{\overline #1}}

\newcommand{\ds}{\displaystyle}
\newcommand {\bR}{\mathbf{R}}
\def\mg{\mathfrak{g}}
\def\mh{\mathfrak{h}}
\newcommand{\id}[1]{\mathop{\rm id}\nolimits_{#1}}

\newtheorem{theorem}{Theorem}
\newtheorem{propos}{Proposition}
\newtheorem{lemma}{Lemma}
\newtheorem{corol}{Corollary}

\begin{document}

\title{Some applications of differential geometry in the theory of mechanical systems\footnote{Submitted on July 29, 1977.}}

\author{M.P.\,Kharlamov\footnote{Donetsk Physico-Technical Institute.}}

\date{}

\maketitle

\begin{center}
{\bf Published: \ \ \textit{Mekh. Tverd. Tela}, 1979, No. 11, pp. 37--49}\footnote{Russian Journal ``Mechanics of Rigid Body''.}

\vspace{3mm}

\href{http://www.ams.org/mathscinet-getitem?mr=536269}{http://www.ams.org (Reference)}

\vspace{1mm}

\href{http://www.ics.org.ru/doc?pdf=1157\&dir=r}{http://www.ics.org.ru (Russian)}

\vspace{1mm}

\href{https://www.researchgate.net/publication/253671029}{https://www.researchgate.net (Russian)}

\end{center}

\begin{abstract}
In the present paper, some concepts of modern differential geometry are used as a basis to develop an invariant theory of mechanical systems, including systems with gyroscopic forces. An interpretation of systems with gyroscopic forces in the form of flows of a given geodesic curvature is proposed. For illustration, the problem of the motion of a rigid body about a fixed point in an axially symmetric force field is examined. The form of gyroscopic forces of the reduced system is calculated. It is shown that this form is a product of the momentum constant, the volume form of the 2-sphere, and an explicitly written everywhere positive function on the sphere.
\end{abstract}

\vspace{3mm}

\section{Introduction}
Qualitative investigation of the problems in classical mechanics uses, during the last years, a widening area of mathematical disciplines. This, in turn, supposes the high level of formalization in the description of corresponding mechanical systems. Such level is already achieved in Hamiltonian mechanics, and the abstract theory of Hamiltonian systems gave a lot of perfect results. As far as the Lagrangian mechanics is concerned, its contemporary presentations \cite{Godbil,Abra} sometimes become too huge when applied to concrete systems. Moreover in mechanics, there exist systems which globally do not have a quadratic Lagrangian \cite{Kh1976,KhFuncAn}. Thus we come to a necessity, on one hand, to simplify and, on the other hand, to generalize the basic notions of the theory. Such an attempt is made in this article. Note that the presence of local coordinates in some proofs is not inevitable; in fact, one can make all reasoning invariant (not using any coordinates).

Basing on a given formalism, we describe the reduction in mechanical systems with symmetry. At the same time, some global relations between the used objects of the differential geometry are revealed. With this approach, the reduced system is still interpreted as a mechanical one.

\section{Natural systems} A natural mechanical system is a triple
\begin{equation}\label{eq01}
  (M,m,V),
\end{equation}
where $M$ is a smooth manifold (the configuration space of the system), $m$ is a Riemannian metric on $M$, and $V$ is a function on $M$ (the potential function of the system, or shortly, the potential). The metric $m$ generates on the tangent space $T(M)$ the function
\begin{equation}\label{eq02}
  K(w)=\frac{1}{2}\bb{w,w}, \quad w\in T(M).
\end{equation}
Here $\bb{,}$ denotes the scalar product in the metric $m$. The function $K$ is called the kinetic energy of system \eqref{eq01}.

For any manifold $N$ we denote by $p_N:T(N)\to N$ the projection to the base of a tangent bundle. The total energy of \eqref{eq01} is the function $H$ on $T(M)$ defined as
\begin{equation}\label{eq03}
  H=K+V\circ p_M.
\end{equation}

Let $w\in T(M)$. In the tangent space $T_w T(M)$ we define the linear form $\tq_w$ by putting for each $X\in T_w(T(M))$
\begin{equation}\label{eq04}
\tq_w(X)=\bb{w,Tp_M(X)}.
\end{equation}

\begin{propos}\label{prop1}
The map
\begin{equation}\label{eq05}
  \tq: w\mapsto \tq_w
\end{equation}
defines a differential $1$-form on $T(M)$. Its exterior derivative
\begin{equation}\label{eq06}
  \sg=\rmd \tq
\end{equation}
makes $T(M)$ a symplectic manifold.
\end{propos}
\begin{proof} For natural coordinates \cite{Abra,Kh1976} $(\bq, \dot \bq)$ in a neighborhood of $w\in T(M)$ the differential forms $(\rmd \bq,\rmd \dot \bq)$ give a basis in each overlying fiber of the cotangent space $T^*(T(M))$. If $A=\|a_{ij}(\bq)\|$ is the definitely positive symmetric matrix of the metric $m$
\begin{equation*}
  m_{\bq}(\dot \bq_{(1)},\dot \bq_{(2)})=a_{ij}(\bq) \dot q_{(1)}^i \dot q_{(2)}^j,
\end{equation*}
then according to \eqref{eq04} the map \eqref{eq05} has the form
\begin{equation}\label{eq07}
  \tq_{(\bq,\dot \bq)} = a_{ij}(\bq)\dot q^i \rmd q^j,
\end{equation}
and, in particular, is smooth. Hence, \eqref{eq05} is a smooth section of $T^*(T(M))$, i.e., a differential 1-form.

Recall that a symplectic structure on a manifold is a closed non-degenerate 2-form on it. Obviously, \eqref{eq06} is closed. Applying the exterior derivative to \eqref{eq07}, we obtain
\begin{equation*}
  \sg_{(\bq, \dot \bq)}=a_{ij}(\bq)\rmd \dot q^i \wedge \rmd \dot q^j+ \dot q^i\frac{\partial a_{ij}(\bq)}{\partial q^k} \rmd q^k \wedge \rmd q^j,
\end{equation*}
therefore the matrix of the form $\sg$
\begin{equation}\label{eq08}
  S=\begin{Vmatrix} * & -A \\ A & 0 \end{Vmatrix}
\end{equation}
has a non-zero determinant $\det S =(\det A)^2$. Hence $\sg$ is non-degenerate.
\end{proof}

The differential forms $\tq$ and $\sg$ defined by \eqref{eq04}\,--\,\eqref{eq06} will be called the Lagrange forms on $T(M)$ generated by the metric $m$. Note that in terms of the book \cite{Abra}, the Lagrange forms generated by the Legendre transformation $\mathbf{F} K$ of the function \eqref{eq02} are $\tq$ and $(-\sg)$.

Let $\rmi_Y\oq$ denote the inner product of a vector field $Y$ and a form $\oq$. For the function \eqref{eq03} and the non-degenerate form \eqref{eq06} there exists a unique vector field $X$ on $T(M)$ such that
\begin{equation}\label{eq09}
  \rmi_X\sg =-\rmd H.
\end{equation}
The field $X$ is a second-order equation \cite{Abra,Godbil} $Tp_M\circ X=\id{T(M)}$, so each integral curve $w(t)$ of the field $X$ is the derivative of its projection
\begin{equation*}
  w(t)=(p_M\circ w)'(t).
\end{equation*}
For an integral curve $w(t)$ of the field $X$, we call $x(t)=p_M\circ w(t)$ a \textbf{motion} in system~\eqref{eq01}.

\begin{propos}\label{prop2}
The total energy \eqref{eq03} is a first integral of system \eqref{eq01}.
\end{propos}

Indeed, according to \eqref{eq09}, $H$ is the Hamilton function for the field $X$ on the symplectic manifold $(T(M),\sg)$.

The Maupertuis principle gives the following geometric interpretation of motions in a natural system.
\begin{theorem}[\cite{Arnold}]\label{theo1}
Let $h$ be a constant such that the region $M_h=\{x\in M: V(x)<h\}$ is not empty. Define the Riemannian metric $m_h=2(h-V)m$ in $M_h$. Then the motions in the natural system \eqref{eq01} having the energy constant $h$ are geodesics of the metric $m_h$.
\end{theorem}

\section{Gyroscopic forces} Along with natural systems, in mechanics we often come across the systems having forces that does not produce work. The existence of such forces, called gyroscopic, is usually expressed by the linear in velocities terms of Lagrangians. But these terms in the general case are defined only locally and up to adding some linear function generated by a closed 1-form. Therefore gyroscopic forces are naturally defined by some 2-form on the configuration space.

A \textbf{mechanical system with gyroscopic forces} is a 4-tuple
\begin{equation}\label{eq10}
  (M,m,V,\vk),
\end{equation}
where $M$ is a manifold, $m$ a Riemannian metric on $M$, $V$ a function on $M$ and $\vk$ a closed 2-form on $M$. All objects are supposed smooth.  Again $M$ is the configuration space, $V$ the potential. The function \eqref{eq02} is the kinetic energy and \eqref{eq03} the total energy of the system.

Let $\tq$ and $\sg$ be the Lagrangian forms generated by the metric $m$.
\begin{propos}\label{prop3}
There exists a unique vector field $X$ on $T(M)$ such that
\begin{equation}\label{eq11}
  \rmi_X(\sg+p_M^*\vk)=-\rmd(K+V\circ p_M).
\end{equation}
This field is a second-order equation.
\end{propos}

For the proof we note that the image of the form $\vk$ under the pull-back $p_M^*$ does not contain, in  local representation, forms of the type $\rmd \dot q^i$, therefore the matrix of the form $\sg+p_M^*\vk$ differs from \eqref{eq08} only in the left upper block and thus has a non-zero determinant. Hence, the form $\sg+p_M^*\vk$ is non-degenerate and the field $X$ exists and is unique. The second assertion is easily checked in natural coordinates on $T(M)$.

The vector field $X$ is the dynamical system corresponding to the mechanical system \eqref{eq10}. The same as above we call $x(t)=p_M\circ w(t)$ a motion in system \eqref{eq10} if $w(t)$ is an integral curve of the field $X$.

\begin{propos}\label{prop4}
The total energy \eqref{eq03} is a first integral of system \eqref{eq10}.
\end{propos}
This fact immediately follows from the definition \eqref{eq11}.

Geometric interpretation of motions in a system with gyroscopic forces can be obtained in the following way.

Let $h\in \bR$ be the value of the energy integral
\begin{equation}\label{eq12}
  K+V\circ p_M=h
\end{equation}
such that the region of possible motions $M_h=\{x\in M: V(x)<h\}$ is not empty.
Define the Riemannian metric $m_h=2(h-V)m$ in $M_h$. Denote by $\Pi_h$ the operator taking 1-forms on $M_h$ to vector fields by the rule
\begin{equation}\label{eq13}
  m_h(\Pi_h(\lqq),Y)=\lqq(Y).
\end{equation}

\begin{theorem}\label{theo2}
Let $x(t)$ be a motion in system \eqref{eq10} satisfying the integral condition \eqref{eq12}. If we denote by $\ov{x}(\tau)$ the same curve but parameterized by the arclength $\tau$ of the metric $m_h$, then
\begin{equation}\label{eq14}
  \frac{D}{d\tau} \frac{d\ov{x}}{d\tau}=-\Pi_h\left(\rmi_{\frac{d\ov{x}}{d\tau}} \vk\right),
\end{equation}
where the covariant derivative is calculated in the metric $m_h$. Inversely, let $\ov{x}(\tau)$ be a curve parameterized by the arclength $\tau$ of the metric $m_h$ and satisfying \eqref{eq14}. Then there exists a change of the parameter $\tau=\tau(t)$ such that the curve $x(t)=\ov{x}(\tau(t))$ is a motion in system \eqref{eq10} on which the condition \eqref{eq12} holds.
\end{theorem}
\begin{proof}
Let us start with the second assertion. It obviously has local character, therefore we use some coordinates $\bq=(q^1,\ldots,q^m)$ on $M$. Let $a_{ij},\Gamma^i_{jk}$ and $\ov{a}_{ij},\ov{\Gamma}^i_{jk}$ be the metric tensor and the Christoffel symbols of $m$ and $m_h$ respectively. Let $\vk=\vk_{ij}\rmd q^i \wedge \rmd q^j$. Denote $\ov{\vk}_{ij}=\vk_{ij}-\vk_{ji}$.

Suppose that the curve $\ov{x}(\tau)=(q^1(\tau),\ldots,q^m(\tau))$ satisfies \eqref{eq14}. In local representation
\begin{equation}\label{eq15}
  \frac{d^2q^i}{d\tau^2}+\ov{\Gamma}^i_{jk}\frac{d q^j}{d\tau}
  \frac{d q^k}{d\tau}=\ov{a}^{ik}\ov{\vk}_{kj}\frac{d q^j}{d\tau},
\end{equation}
where $\ov{a}^{ij}\ov{a}_{jk}=\delta^i_k$. By definition, the following relations hold
\begin{equation}\label{eq16}
  \ov{a}_{ij}=2(h-V)a_{ij}, \qquad a^{ij}=2(h-V)\ov{a}^{ij}.
\end{equation}
Substituting \eqref{eq16} into \eqref{eq15} we obtain
\begin{equation}\label{eq17}
  \begin{array}{l}
\ds     \frac{d^2q^i}{d\tau^2}+\Gamma^i_{jk}
     \frac{d q^j}{d\tau}
  \frac{d q^k}{d\tau}+\frac{a^{i\ell}}{2(h-V)}\left[ a_{\ell j}\frac{\partial (h-V)}{\partial q^k}+a_{\ell k}\frac{\partial (h-V)}{\partial q^j} - \right.\\
\ds \qquad \left.  - a_{j k}\frac{\partial (h-V)}{\partial q^\ell}\right]\frac{d q^j}{d\tau}
  \frac{d q^k}{d\tau}=\frac{a^{ik}}{2(h-V)} \ov{\vk}_{kj}\frac{d q^j}{d\tau}.
  \end{array}
\end{equation}

By assumption $\ov{x}(\tau)\in M_h$, then $h-V(\ov{x}(\tau))>0$ and we can make the following monotonous change of the parameter
\begin{equation}\label{eq18}
\ds  dt=\frac{d\tau}{2(h-V(\ov{x}(\tau)))}.
\end{equation}
Since $\tau$ is the natural parameter on $\ov{x}$, we have
\begin{equation}\label{eq19}
  2(h-V)a_{ij}\frac{d q^i}{d\tau}
  \frac{d q^j}{d\tau}=1.
\end{equation}
Then applying the change \eqref{eq18} to \eqref{eq17} we obtain the equation
\begin{equation*}
  \frac{d^2q^i}{dt^2}+\Gamma^i_{jk}\frac{d q^j}{d t}
  \frac{d q^k}{dt}+a^{ij}\frac{\partial V}{\partial q^j}=a^{ik}\ov{\vk}_{kj}\frac{d q^j}{dt},
\end{equation*}
which is a local representation of the fact that $x(t)=\ov{x}(\tau(t))$ is a solution of the second-order equation $X$ defined according to \eqref{eq11}. The conservation law \eqref{eq12} holds due to the choice of the change \eqref{eq18} and the condition \eqref{eq19}.

The proof of the first assertion can be obtained by making all substitutions in reversed order. The variational proof for the case of 2-dimensional $M$ can be found in~\cite{AnSin}.
\end{proof}

The flows on iso-energetic manifolds defined by \eqref{eq14} can naturally be called the \textbf{flows of given curvature}. If $\vk \equiv 0$, we obtain usual geodesic flows.

The case when $\dim M=2$ is essentially special because in this case the geodesic curvature of trajectories depends only on the point of $M$ rather than on the direction of trajectories. Namely, let $o_h$ be the volume form on $M_h$ corresponding to the metric $m_h$. Then there exists a function $k_h$ on $M_h$ such that $\vk=k_h o_h$. A simple calculation shows that for each vector $v$ tangent to $M_h$ at some point $x$ and having length 1 in the metric $m_h$, the vector $w=-\Pi_h(\rmi_v \vk)$ is orthogonal to $v$ in the metric $m_h$ and its length is $\|w\|_{m_h}=|k_h(x)|$. According to \eqref{eq13},
\begin{equation*}
  \begin{array}{rcl}
     \vk(w,v)& = &\vk(v,\Pi_h(\rmi_v \vk))=(\rmi_v \vk)(\Pi_h(\rmi_v \vk))= \\
     {}&  = & m_h(\Pi_h(\rmi_v \vk), \Pi_h(\rmi_v \vk))=k_h^2>0.
   \end{array}
\end{equation*}
This means that the basis $\{w,v\}$ in $T_x(M)$ defines in $T_x(M)$ the same orientation as the 2-form $\vk$. Therefore, for 2-dimensional systems we proved the following statement.
\begin{propos}\label{prop5}
Let in \eqref{eq10} $\dim M=2$. A curve $x(t)$ satisfying \eqref{eq12} is the motion in system \eqref{eq10} if and only if, being parameterized by the arclength of the metric $m_h$, it has the geodesic curvature $|\vk/o_h|$ and the basis $\{$\emph{the curvature vector, the tangent vector}$\}$ gives the same orientation of $T_{x(t)}(M)$ as the form $\vk$.
\end{propos}

\section{Invariant theory of reduction\\ in systems with symmetry}
Let us suppose that a one-parameter group $G=\{g^\tau\}$ acts as diffeomorphisms of the configurational space of the natural mechanical system \eqref{eq01} and this action generates a principal $G$-bundle \cite{Bishop}
\begin{equation}\label{eq20}
  \mB=(M,p,\tl{M}),
\end{equation}
where $\tl{M}=M/G$ is the quotient manifold and $p:M \to \tl{M}$ the factorization map. Suppose also that all $g^\tau$ preserve the metric $m$ and the potential $V$. Obviously, diffeomorphisms from the group $G_T=\{Tg^\tau: g^\tau \in G\}$ preserve the kinetic energy \eqref{eq02} of \eqref{eq01} and, consequently, the total energy \eqref{eq03}. For the generating vector fields
\begin{eqnarray}
  v(x) &=& {\left.\frac{d}{d\tau}\right|}_{\tau=0}g^\tau(x), \qquad x\in M, \label{eq21}\\
  v_T(w) &=& {\left.\frac{d}{d\tau}\right|}_{\tau=0}Tg^\tau(w), \qquad w\in T(M) \label{eq22}
\end{eqnarray}
we have
\begin{equation}\label{eq23}
  vV\equiv 0, \qquad v_T K\equiv 0, \qquad v_T H\equiv 0.
\end{equation}
The group $G$ satisfying \eqref{eq23} is called the symmetry group of system \eqref{eq01}. The theory of natural systems with symmetries was created by S.\,Smale \cite{Smale}. The starting point of it is the momentum integral. For a one-parameter group let us use a simpler definition of the momentum \cite{Ta1973} connected with Noether's theorem \cite{Arnold}.

The momentum of a mechanical system \eqref{eq01} with symmetry $G$ is the function $J$ on $T(M)$ defined as
\begin{equation}\label{eq24}
  J(w)=\bb{v,w},
\end{equation}
where $v$ is the vector field \eqref{eq21}. It is clear that $J$ is everywhere regular and $G_T$-invariant. Let us show that it is a first integral of system \eqref{eq01}.
\begin{lemma}\label{lem1}
The Lagrange forms $\tq$ and $\sg$ generated by the metric $m$ are preserved by the group $G_T$.
\end{lemma}
\begin{proof}
For all $w\in T(M),Y\in T_w(T(M))$ we have
\begin{equation}\label{eq25}
\begin{array}{l}
  \tq_{Tg^\tau(w)}(TTg^\tau(Y))=\bb{Tg^\tau(w),Tp_M\circ TTg^\tau(Y)}=\\
  \qquad =\bb{Tg^\tau(w),Tg^\tau\circ Tp_M(Y)}=\bb{w,Tp_M(Y)}.
\end{array}
\end{equation}
The last equality follows from the fact that $g^\tau$ are isometries of $m$. Equations \eqref{eq04} and \eqref{eq25} yield $(Tg^\tau)^*\tq=\tq$, hence, according to \eqref{eq06}, $(Tg^\tau)^*\sg=\sg$. \end{proof}

\begin{corol}\label{thecor1} The field $X$ defining dynamics of system \eqref{eq01} is preserved by the group $G_T$, i.e., for all $g^\tau\in G$
\begin{equation*}
  TTg^\tau\circ X=X\circ Tg^\tau.
\end{equation*}
The generating field \eqref{eq22} commutes with $X$:
\begin{equation}\label{eq26}
  [v_T,X]\equiv 0.
\end{equation}
\end{corol}

The proof follows immediately from the invariance of $H$ and definition \eqref{eq09}.

Now let us note that the fields \eqref{eq21} and \eqref{eq22} satisfy
\begin{equation}\label{eq27}
  Tp_M\circ v_T=v.
\end{equation}
Therefore, using definition \eqref{eq04}, we can calculate the derivative of the momentum along $X$ as $XJ=X\tq(v_T)$. Let us add to the right-hand part the terms $\tq([v_T,X])$ and $-v_T\tq(X)=-2 v_T K$ equal to zero in virtue of \eqref{eq23}, \eqref{eq26} and use the rule for the exterior derivative of a 1-form \cite{Godbil}\footnote[1]{$\rmd \alpha (U,V)=\alpha([U,V])+U \alpha(V)-V\alpha(U)$.}. Then we obtain
\begin{equation*}
  XJ=\rmd \tq(X,v_T)=-\rmd H(v_T)=-v_T H \equiv 0.
\end{equation*}
Here we used \eqref{eq23} for $v_T H$ and definition \eqref{eq09}. Thus, the momentum $J$ is a first integral of the field $X$. In particular, for any $k\in \bR$ the set $J_k=J^{-1}(k)$ is a $G_T$-invariant integral submanifold in $T(M)$ of codimension 1.

It is clear that $J_k(x)=J_k \cap T_x(M)$ is a hyperplane in $T_x(M)$ and it contains zero if and only if $k=0$. The subspace $J_0(x)$ is the orthogonal supplement, in metric $m$, to the line $T_x^v$ spanned by the generating vector \eqref{eq21}. The hyperplane $J_k(x)$ is parallel to $J_0(x)$ and therefore the intersection $T_x^v \cap J_k(x)$ consists of a unique vector
\begin{equation}\label{eq28}
  v^k(x)=k v/\bb{v,v}.
\end{equation}
The vector field $v^k$ is smooth and $G_T$-invariant.

\def\U{\Gamma}

The set of subspaces $J_0(x)$ generates a connexion \cite{Bishop} in the principal $G$-bundle \eqref{eq20}. Let $\mh$ be the form of the connexion $J_0$,
\begin{equation}\label{eq29}
  \mh(w)=\bb{v,w}/\bb{v,v},
\end{equation}
and $\mg$ the corresponding curvature form,
\begin{equation}\label{eq30}
  \mg =\rmd \mh
\end{equation}
(here we have the standard exterior derivative since $G$ is commutative). Let $\U_k: J_0 \to J_k$ be the diffeomorphism defined by
\begin{equation}\label{eq31}
  \U_k(w)=w+v^k(x), \qquad w\in J_0(x).
\end{equation}
Denote by $\tq^k$ and $\sg^k$ the differential forms induced on $J_k$ by the Lagrange forms of the metric $m$ under the embedding $J_k \subset T(M)$.

\begin{propos}\label{prop6}
The following equalities hold
\begin{eqnarray}
  \U_k^*\tq^k &=& \tq^0+k\, p_M^*\mh, \label{eq32}\\
  \U_k^*\sg^k &=& \sg^0+k\, p_M^*\mg. \label{eq33}
\end{eqnarray}
\end{propos}
\begin{proof}
Let $w\in J_0, Y\in T_w(J_0)$. From \eqref{eq04},
\begin{equation}\label{eq34}
  \tq^0_w(Y)=\bb {w,Tp_M(Y)}.
\end{equation}
On the other hand,
\begin{equation*}
\begin{array}{l}
  (  \U_k^*\tq^k)_w(Y)=\tq^k_{\U_k(w)}(T\U_k(Y))= \bb{\U_k(w),Tp_M\circ T\U_k(Y)} =\\
  \qquad = \bb{w+v^k,T(p_M\circ \U_k)(Y)  }=\bb{w,Tp_M(Y)}+\bb{v^k,Tp_M(Y)}.
\end{array}
\end{equation*}
Here we used the identity $p_M\circ \U_k = p_M$ from \eqref{eq31}. From \eqref{eq28} and \eqref{eq29} we have
\begin{equation*}
  \bb{v^k,Tp_M(Y)}=k\, \mh(Tp_M(Y)),
\end{equation*}
therefore,
\begin{equation}\label{eq35}
  (  \U_k^*\tq^k)_w(Y)= \bb{w,Tp_M(Y)}+k \,p_M^*\mh (Y).
\end{equation}
Comparing \eqref{eq34} with \eqref{eq35}, we obtain \eqref{eq32}. Now \eqref{eq33} follows from \eqref{eq06} and \eqref{eq32} since the exterior derivative commutes with pull-back mappings of forms.
\end{proof}

Let us introduce the map
\begin{equation*}
  \ro_k:J_k \to T(\tl{M})
\end{equation*}
as the restriction to $J_k$ of the map $Tp:T(M)\to T(\tl{M})$. Using an atlas of the bundle $\mB$ one can show that the triple
\begin{equation}\label{eq36}
  \mB_k=(J_k,\ro_k,T(\tl{M}))
\end{equation}
is a principle $G_T$-bundle.
\begin{theorem}\label{theo3}
The forms $\tq^k$ and $\sg^k$ are preserved by the group $G_T$. The form $\sg^k$ is horizontal in the sense of the bundle \eqref{eq36}. The form $\tq^k$ is horizontal if and only if $k=0$.
\end{theorem}
\begin{proof}
The first assertion follows from Lemma~\ref{lem1}. The form $\tq^k$ is horizontal if for all $w\in J_k$
\begin{equation*}
  \tq_w^k(v_T)=0.
\end{equation*}
This in virtue of \eqref{eq04} and \eqref{eq27} means that $\bb{w,v}=0$ for all $w\in J_k$. But $v^k\in J_k$ and $\bb{v^k,v}=k$. So the form $\tq^k$ is horizontal only for $k=0$. The fact that $\sg^k$ is horizontal follows from the structural equation for horizontal forms \cite{Bishop} and the fact that $G_T$ is commutative.

Note that the field $v_T$ is preserved by diffeomorphisms \eqref{eq31}
\begin{equation*}
  T\U_k\circ v_T=v_T \circ \U_k.
\end{equation*}
Hence, in virtue of \eqref{eq33},
\begin{equation}\label{eq37}
  \U_k^*\rmi_{v_T}\sg^k=\rmi_{v_T}\U_k^*\sg^k=\rmi_{v_T}\sg^0+k\,\rmi_{v_T}p_M^*\mg.
\end{equation}
The first term in the right-hand part is zero because $\sg^0$ is horizontal. For the second term we get
\begin{equation*}
  k\,\rmi_{v_T}p_M^*\mg =k\,p_M^*\rmi _{Tp_M(v_T)}\mg=k\,p_M^*\rmi_v \mg=0,
\end{equation*}
since the curvature form $\mg$ is horizontal in the sense of the bundle \eqref{eq20}. The theorem is proved.
\end{proof}

As a corollary of Theorem~\ref{theo3} we get the existence of differential forms $\tl{\tq}^0$ and $\tl{\sg}^k$ such that $\tq^0=\ro_0^*\tl{\tq}^0, \sg^k=\ro_k^*\tl{\sg}^k$. In turn, for the form \eqref{eq30} we have $\mg=p^*\tl{\mg}$ for some 2-form $\tl{\mg}$ on $\tl{M}$. Then from \eqref{eq33} we obtain
\begin{equation*}
  \U_k^*\ro_k^*\tl{\sg}^k=\ro_0^*\tl{\sg}^0+k\,p_M^* p^* \tl{\mg}.
\end{equation*}
Whence, having the obvious equalities $\ro_k \circ \U_k=\ro_0$ and $p\circ p_M=p_{\tl{M}}\circ Tp$, \begin{equation}\label{eq38}
  \tl{\sg}^k=\tl{\sg}^0+k\,p_{\tl{M}}^* \tl{\mg}.
\end{equation}

Let us define a Riemannian metric $\tl{m}$ on $\tl{M}$ putting for every $\tl{w}_1,\tl{w}_2\in T_{\tl{x}}(\tl{M})$
\begin{equation}\label{eq39}
  \tl{m}(\tl{w}_1,\tl{w}_2)=\bb{w_1,w_2},
\end{equation}
where $w_1,w_2\in J_0(x)$ are chosen to give $\ro_0(w_i)=\tl{w}_i$ (in particular, $p(x)=\tl{x}$).

\begin{propos}\label{prop7}
The differential forms $\tl{\tq}^0$ and $\tl{\sg}^0$ are the Lagrange forms on $T(\tl{M})$ generated by the metric $\tl{m}$.
\end{propos}

\begin{proof}
According to \eqref{eq04} it is sufficient to show that for all $\tl{Y}\in T_{\tl{w}}(T(\tl{M}))$
\begin{equation}\label{eq40}
  \tl{\tq}^0_{\tl{w}}(\tl{Y})=\tl{m}(\tl{w},Tp_{\tl{M}}(\tl{Y})).
\end{equation}
Take $w\in J_0$ and $Y\in T_w(J_0)$ such that $\ro_0(w)=\tl{w}$ and $T\ro_0(Y)=\tl{Y}$. We can write
\begin{equation}\label{eq41}
  Tp_M(Y)=w^0+c \,v,
\end{equation}
where $w^0\in J_0$ and $v$ is the vector \eqref{eq21}. Since $v$ is orthogonal to $J_0$, we have
\begin{equation*}
  \tq_w^0(Y)=\bb{w,Tp_M(Y)}=\bb{w,w^0}.
\end{equation*}
But, according to \eqref{eq39}, $\bb{w,w^0}=\tl{m}(\ro_0(w),\ro_0(w^0))$. Then in virtue of $Tp(v)=0$ we get from \eqref{eq41}
\begin{equation*}
  \ro_0(w^0)=Tp\circ Tp_M(Y)=T(p_{\tl{M}}\circ Tp)(Y)=Tp_{\tl{M}}\circ T\ro_0(Y)=Tp_{\tl{M}}(\tl{Y}).
\end{equation*}
This yields \eqref{eq40}.
\end{proof}

\begin{corol}\label{thecor2}
The pair $(T(\tl{M}),\tl{\sg}^k)$ is a symplectic manifold.
\end{corol}

\vskip2mm
\textbf{Definition}. The \textbf{reduced system} corresponding to the momentum value $k$ is a vector field $\tl{X}_k$ on $T(\tl{M})$ such that on $J_k$ the following identity holds
\begin{equation}\label{eq42}
  T\ro_k \circ X=\tl{X}_k\circ \ro_k.
\end{equation}
\vskip2mm

According to Corollary~\ref{thecor1}, the field $\tl{X}_k$ exists and is unique. It follows from \eqref{eq42} that the set of its integral curves is the $\ro_k$-image of the set of integral curves of the field $X$ with the momentum $k$.

Denote by $H_k$ the restriction of the total energy $H$ of system \eqref{eq01} to the submanifold $J_k$. Since $H$ is $G_T$-invariant, there exists a unique function $\tl{H}_k$ (\textbf{the reduced energy}) satisfying the relation
\begin{equation}\label{eq43}
  H_k=\tl{H}_k\circ \ro_k.
\end{equation}
It is easily shown that
\begin{equation}\label{eq44}
  \tl{H}_k=\tl{K}+\tl{V}_k\circ p_{\tl{M}},
\end{equation}
where $\tl{K}(\tl{w})=\frac{1}{2} \tl{m}(\tl{w},\tl{w})$ is the kinetic energy of the reduced metric $\tl{m}$ and the function $\tl{V}_k$ on $\tl{M}$ (called the \textbf{amended or effective potential}) is defined by
\begin{equation}\label{eq45}
  \tl{V}_k(p(x))=V(x)+\frac{k^2}{2\bb{v(x),v(x)}}.
\end{equation}

\begin{theorem}\label{theo4}
The reduced system $\tl{X}_k$ is a Hamiltonian field on the symplectic manifold $(T(\tl{M}),\tl{\sg}^k)$ with the Hamilton function equal to the reduced energy.
\end{theorem}

Indeed, from \eqref{eq09}, \eqref{eq42}, and \eqref{eq43} we obtain
\begin{equation}\label{eq46}
  \rmi_{\tl{X}_k} \tl{\sg}^k = - \rmd \tl{H}_k,
\end{equation}
and this is a definition of the Hamiltonian field for $\tl{H}_k$.

\begin{theorem}\label{theo5}
The reduced system $\tl{X}_k$ is the dynamical system corresponding to the mechanical system with gyroscopic forces
\begin{equation}\label{eq47}
  (\tl{M},\tl{m},\tl{V}_k,k\,\tl{\mg}),
\end{equation}
where the 2-form $\tl{\mg}$ is induced by the curvature form of the connexion $J_0$ in the principal bundle $(M,p,\tl{M})$.
\end{theorem}
\begin{proof}
According to \eqref{eq38}, \eqref{eq44}, and \eqref{eq46} we have
\begin{equation*}
  \rmi _{\tl{X}_k}(\tl{\sg}^0+p_{\tl{M}}^*(k\,\tl{\mg}))=-\rmd(\tl{K}+\tl{V}_k \circ p_{\tl{M}}),
\end{equation*}
so the assertion of the theorem follows from Proposition~\ref{prop7} and definition \eqref{eq11}.
\end{proof}

\begin{corol}\label{thecor3}
The reduced system $\tl{X}_k$ is a second-order equation on $\tl{M}$. A curve $\tl{x}(t)$ in $\tl{M}$ is a motion in system \eqref{eq47} if and only if $\tl{x}(t)=p\circ x(t)$, where $x(t)$ is a motion in system~\eqref{eq01} with the momentum $J(x'(t))=k$.
\end{corol}

\section{Reduced system in rigid body dynamics} The problem of the motion of a rigid body having a fixed point in the axially symmetric force field (e.g. the gravity field or the field of a central Newtonian force) with an appropriate choice of variables has a cyclic coordinate and admits the reduction by Routh method. However, as shown in \cite{KhFuncAn}, this method can be applied only locally, and this fact is not connected with singularities of local coordinate systems, but reflects the essence of the problem as a whole. The above described approach makes it possible to describe the reduced system globally in terms of the redundant variables (direction cosines), which are applicable in the same way everywhere on the reduced configuration space.

Suppose that the body is fixed in its point $O$ at the origin of the cartesian coordinate system $O\bi_1\bi_2\bi_3$ of the inertial space $\bR^3$. The components of vectors from $\bR^3$ in the basis $\mathbf{\bi}=\|\bi_1,\bi_2,\bi_3\|$ will be written in a column. Let the unit vectors $\be_1,\be_2,\be_3$ go along the principal inertia axes at $O$ and $I_1,I_2,I_3$ be the corresponding principal moments of inertia. The row $\mathbf{\be}=\|\be_1,\be_2,\be_3\|$ is an orthonormal basis in $\bR^3$.

To any position $\mathbf{\be}$ of the body we assign the matrix $Q\in SO(3)$ such that
\begin{equation}\label{eq48}
  \mathbf{\bi}Q=\mathbf{\be}.
\end{equation}
It is clear that the map $\mathbf{\be}\mapsto Q$ is one-to-one and the group $SO(3)$ can be considered as the configuration space of a rigid body with a fixed point \cite{Arnold}.  The Lie algebra of $SO(3)$ (the tangent space at the unit) is the 3-dimensional space $\mathfrak{so}(3)$ of skew-symmetric $3{\times}3$ matrices with the standard commutator
\begin{equation}\label{eq49}
  [\Omega_1,\Omega_2]=\Omega_1 \Omega_2-\Omega_2 \Omega_1.
\end{equation}
Obviously, for any $Q\in SO(3)$
\begin{equation}\label{eq50}
  T_Q(SO(3))=Q\,\mathfrak{so}(3) = \mathfrak{so}(3)\,Q.
\end{equation}

We fix an isomorphism $f$ of the vector spaces $\mathfrak{so}(3)$ and $\bR^3$. Namely,
\begin{equation*}
  \Omega=\begin{Vmatrix} 0 & -\oq_3&\oq_2\\
                         \oq_3 & 0 & -\oq_1\\
                         -\oq_2 & \oq_1 & 0
                         \end{Vmatrix} \quad \mapsto \quad f(Q)=\begin{Vmatrix} \oq_1\\
                                                                                \oq_2\\
                                                                                \oq_3                         \end{Vmatrix}.
\end{equation*}
It is shown straightforwardly that $f$ takes the commutator \eqref{eq49} to the standard cross product
\begin{equation}\label{eq51}
  f([\Omega_1,\Omega_2])=f(\Omega_1)\times f(\Omega_2).
\end{equation}

The tangent bundle of the Lie group is trivial. One of the possible trivializations of $T(SO(3))$ is given by the map
\begin{equation}\label{eq52}
  T(SO(3))\to SO(3)\times \bR^3: (Q,\dot Q) \mapsto (Q,f(Q^{-1}\dot Q)),
\end{equation}
which is well defined in virtue of \eqref{eq50}. For the sake of being short, we call the vector
\begin{equation}\label{eq53}
  \oq=f(Q^{-1}\dot Q) \in \bR^3
\end{equation}
the spin of the rotation velocity $\dot Q$, although it is a slight abuse of terminology. Let us describe the mechanical sense of it. Differentiating \eqref{eq48}, we obtain
\begin{equation}\label{eq54}
 \dot {\mathbf{\be}} = \mathbf{\bi}\dot Q =\mathbf{\be}Q^{-1}\dot Q.
\end{equation}
Denote
\begin{equation}\label{eq55}
  \oq=\begin{Vmatrix} \oq_1\\
                      \oq_2\\
                      \oq_3
      \end{Vmatrix}.
\end{equation}
Equation \eqref{eq54} in virtue of definition \eqref{eq53} takes the form
\begin{equation*}
  \dot\be_1=\oq_3\be_2-\oq_2\be_3,\qquad  \dot\be_2=\oq_1\be_3-\oq_3\be_1,\qquad   \dot\be_3=\oq_2\be_1-\oq_1\be_2,
\end{equation*}
i.e., the spin components in the basis $\mathbf{\bi}$ are the projections of the angular velocity vector to the moving axes. The vector defined by \eqref{eq53} is also called the \textbf{angular velocity in the body} \cite{Arnold}. It is clear that the set of all rotation velocities $\dot Q$ with the same spin is a left invariant vector field on $SO(3)$.

Let us consider the one-parameter subgroup $\{Q^\tau\}\subset SO(3)$ consisting of the matrices
\begin{equation*}
  Q^\tau = \begin{Vmatrix} 1 & 0 & 0\\
                           0 & \cos\tau & -\sin\tau \\
                           0 & \sin\tau & \cos\tau
                           \end{Vmatrix}.
\end{equation*}
It acts as a one-parameter group $G=\{g^\tau\}$ of diffeomorphisms of $SO(3)$,
\begin{equation}\label{eq56}
  g^\tau(Q)=Q^\tau Q.
\end{equation}
The generating vector field
\begin{equation*}
  v(Q)=\left. \frac{d}{d\tau}\right|_{\tau=0}g^\tau(Q)=\left. \frac{dQ^\tau}{d\tau}\right|_{\tau=0}Q
\end{equation*}
is right invariant; at the point
\begin{equation}\label{eq57}
Q = \begin{Vmatrix} \aq_1 & \aq_2 & \aq_3\\
                    \aq'_1 & \aq'_2 & \aq'_3 \\
                    \aq''_1 & \aq''_2 & \aq''_3
                           \end{Vmatrix}
\end{equation}
it has the spin
\begin{equation}\label{eq58}
  \nu =f (Q^{-1}\left. \frac{dQ^\tau}{d\tau}\right|_{\tau=0}Q) =\begin{Vmatrix} \aq_1 \\
  \aq_2 \\
  \aq_3\end{Vmatrix}.
\end{equation}
Comparing \eqref{eq48} with \eqref{eq57} and \eqref{eq58} we see that $G$ rotates the body about the fixed in space vector $\bi_3$, the direction of which is usually said to be vertical.

The map $p: Q\mapsto \nu$ defined by \eqref{eq58} takes $SO(3)$ to the unit sphere in $\bR^3$
\begin{equation}\label{eq59}
  \aq_1^2+\aq_2^2+\aq_3^2=1.
\end{equation}
This sphere is called the Poison sphere. The inverse image of each point \eqref{eq59} is exactly an orbit of $G$ and therefore
\begin{equation}\label{eq60}
  p: SO(3) \to S^2
\end{equation}
is the quotient map. Since $p$ is smooth and $G$ is compact, the triple $\mB=(SO(3),p,S^2)$ is a principle $G$-bundle \cite{Arnold}.

The symmetric inertia operator, diagonal in the basis $\mathbf{\bi}$,
\begin{equation*}
  I=\begin{Vmatrix} I_1 & {} & {}\\
  {}&I_2&{}\\
  {}& {}& I_3 \end{Vmatrix}
\end{equation*}
defines the Riemannian metric on $SO(3)$ which in the structure of \eqref{eq52} is
\begin{equation}\label{eq61}
  m_Q(\oq^1,\oq^2)=I\oq^1\cdot \oq^2
\end{equation}
(the dot stands for the standard scalar product in $\bR^3$). The metric \eqref{eq61} is left invariant (since the components of the spin are left invariant) and, in particular, is preserved by the transformations \eqref{eq56}. The corresponding kinetic energy \eqref{eq02} has the classical form
\begin{equation}\label{eq62}
  K=\frac{1}{2} (I_1\oq_1^2+I_2\oq_2^2+I_3\oq_3^2).
\end{equation}

Supposing that the force field has a symmetry axis, we can choose the basis $\mathbf{\bi}$ in such a way that the symmetry axis is the vertical $O\bi_1$. Then the transformations \eqref{eq56} preserve the potential energy $V:SO(3)\to \bR$ and, therefore,
\begin{equation}\label{eq63}
  V=\tl{V}\circ p,
\end{equation}
where $p$ is the map \eqref{eq60} and $\tl{V}=\tl{V}(\aq_1,\aq_2,\aq_3)$ is a function on the sphere \eqref{eq59}.

Thus, the problem of the motion of a rigid body with a fixed point is described by the mechanical system
\begin{equation}\label{eq64}
  (SO(3),m,V)
\end{equation}
with symmetry $G$, where $m$ and $V$ are defined by \eqref{eq61} and \eqref{eq63}, $G$ acts according to \eqref{eq56}.

By Theorem~\ref{theo5}, system \eqref{eq64} generates the mechanical system with gyroscopic forces having the Poisson sphere as the reduced configuration space. Such system obviously defines the motion of the direction vector of the vertical in the coordinate system fixed in the body. Let us calculate the elements of this system.

The momentum corresponding to the symmetry group $G$ is found from \eqref{eq24}, \eqref{eq55}, \eqref{eq58}, and \eqref{eq61},
\begin{equation}\label{eq65}
  J(Q,\dot Q)=I_1\aq_1\oq_1+I_2\aq_2\oq_2+I_3\aq_3\oq_3.
\end{equation}

\begin{lemma}[Ya.V.\,Tatarinov]\label{lem2}
In the product structure \eqref{eq52} the map tangent to \eqref{eq60} is
\begin{equation}\label{eq66}
  Tp(Q,\oq)=p(Q)\times \oq.
\end{equation}
\end{lemma}
\begin{proof}
Denote $\ds{\Omega=  \left.\frac{dQ^\tau}{d\tau}\right|_{\tau=0}\in \mathfrak{so}(3)}$.
By definition
\begin{equation*}
  \begin{array}{l}
\ds     Tp(Q,\oq)=\frac{d}{dt}f(Q^{-1}\Omega Q)=f(Q^{-1}\Omega\dot Q+\dot{Q}^{-1}\Omega Q)= f(Q^{-1}\Omega\dot Q- Q^{-1}\dot{Q}Q^{-1}\Omega Q)=\\
\ds     \qquad = f([Q^{-1}\Omega Q,Q^{-1}\dot Q])=f(Q^{-1}\Omega Q)\times f(Q^{-1}\dot Q) =p(Q)\times \oq.
   \end{array}
\end{equation*}
Here we used the identity $\dot Q ^{-1}Q+Q^{-1}\dot Q \equiv 0$ and the property \eqref{eq51}.
\end{proof}

The tangent map $Tp$ establishes an isomorphism of the horizontal subspace $J_0(Q)$ in $T_Q(SO(3))$ and the tangent plane to the Poisson sphere at the point $p(Q)$,
\begin{equation}\label{eq67}
  \aq_1\dot{\aq}_1+\aq_2\dot{\aq}_2+\aq_3\dot{\aq}_3=0.
\end{equation}
Denote by
\begin{equation*}
  \oq^0 =\begin{Vmatrix} \oq^0_1 \\
  \oq^0_2 \\
  \oq^0_3\end{Vmatrix}
\end{equation*}
the spin of the horizontal vector from $T(SO(3))$ covering the tangent vector $\dot \nu \in T(S^2)$,
\begin{equation*}
  \dot \nu =\begin{Vmatrix} \dot \aq_1 \\
  \dot \aq_2 \\
  \dot \aq_3\end{Vmatrix}.
\end{equation*}
Then from \eqref{eq65} -- \eqref{eq67} we get $\oq^0\cdot I\nu=0$, $\dot\nu =\nu \times \oq^0$, $\nu \cdot \dot\nu =0$. This immediately yields $\ds{\oq^0=\frac{\dot \nu \times I\nu}{I\nu \cdot \nu}}$. In the coordinate form
\begin{equation}\label{eq68}
\begin{array}{c}
\ds  \oq^0_1=\frac{I_3\aq_3 \dot{\aq}_2-I_2\aq_2 \dot{\aq}_3}{I_1\aq_1^2+I_2\aq_2^2+I_3\aq_3^2},\; \ds  \oq^0_2=\frac{I_1\aq_1 \dot{\aq}_3-I_3\aq_3 \dot{\aq}_1}{I_1\aq_1^2+I_2\aq_2^2+I_3\aq_3^2},\;
\ds  \oq^0_3=\frac{I_2\aq_2 \dot{\aq}_1-I_1\aq_1 \dot{\aq}_2}{I_1\aq_1^2+I_2\aq_2^2+I_3\aq_3^2}
\end{array}
\end{equation}
we obtain a partial case of the relations found by G.V.\,Kolosov \cite{Kolosov}.

The latter equations can be considered from another point of view. According to the definition of the spin, its components $\oq_1,\oq_2,\oq_3$ can be treated as 1-forms on $SO(3)$. Then $\oq^0_1,\oq^0_2,\oq^0_3$ are the horizontal parts of these forms. Since $Tp$ induces an isomorphism between horizontal $G$-invariant forms on $SO(3)$ and forms on the Poisson sphere, the formulas
\begin{equation}\label{eq69}
\begin{array}{c}
\ds  \oq^0_1=\frac{I_3\aq_3 \rmd {\aq}_2 -I_2 \aq_2 \rmd{\aq}_3}{I_1\aq_1^2+I_2\aq_2^2+I_3\aq_3^2},\; \ds  \oq^0_2=\frac{I_1\aq_1 \rmd{\aq}_3-I_3\aq_3 \rmd{\aq}_1}{I_1\aq_1^2+I_2\aq_2^2+I_3\aq_3^2},\;
\ds  \oq^0_3=\frac{I_2\aq_2 \rmd{\aq}_1-I_1\aq_1 \rmd{\aq}_2}{I_1\aq_1^2+I_2\aq_2^2+I_3\aq_3^2}
\end{array}
\end{equation}
give the explicit expression of this isomorphism.

The reduced metric on $S^2$ is found from \eqref{eq39}, \eqref{eq62}, and \eqref{eq68},
\begin{equation}\label{eq70}
\tl{m}(\dot \nu,\dot \nu)=I_1(\oq_1^0)^2+I_2(\oq_2^0)^2+I_3(\oq_3^0)^2=\frac{I_1 I_2 I_3 \left(\displaystyle{\frac{\dot{\aq}_1^2}{I_1}+\frac{\dot{\aq}_2^2}{I_2}+\frac{\dot{\aq}_3^2}{I_3}}\right)}{I_1\aq_1^2+I_2\aq_2^2+I_3\aq_3^2}.
\end{equation}
Here we also used equality \eqref{eq67}. It is easily seen that the metric $\tl{m}$ is conform equivalent to the ellipsoidal one \cite{Kolosov}.

Using \eqref{eq29} we find the form $\mh$ of the connexion $J_0$,
\begin{equation}\label{eq71}
  \mh_Q(\dot Q)=\frac{\bb{\dot Q,v(Q)}}{\bb{v(Q),v(Q)}}=\frac{I_1\aq_1\oq_1+I_2\aq_2\oq_2+I_3\aq_3\oq_3}{I_1\aq_1^2+I_2\aq_2^2+I_3\aq_3^2}.
\end{equation}
The exterior derivative of \eqref{eq71} gives the curvature form
\begin{equation}\label{eq72}
  \begin{array}{l}
\ds    \mg=\rmd\frac{1}{I_1\aq_1^2+I_2\aq_2^2+I_3\aq_3^2}\wedge (I_1\aq_1\oq_1+I_2\aq_2\oq_2+I_3\aq_3\oq_3)+ \\[3mm]
\ds   \quad \frac{I_1\aq_1 \rmd\oq_1+I_2\aq_2 \rmd\oq_2+I_3\aq_3 \rmd\oq_3+I_1\rmd\aq_1\wedge \oq_1+I_2\rmd\aq_2\wedge \oq_2+I_3\rmd\aq_3\wedge \oq_3}{I_1\aq_1^2+I_2\aq_2^2+I_3\aq_3^2}.
  \end{array}
\end{equation}

\begin{propos}\label{prop8}
The components of the spin $\oq_1,\oq_2,\oq_3$ considered as $1$-forms on $SO(3)$ satisfy the relations
\begin{equation}\label{eq73}
  \rmd \oq_1=\oq_3\wedge \oq_2, \qquad \rmd \oq_2=\oq_1\wedge \oq_3, \qquad \rmd \oq_3=\oq_2\wedge \oq_1.
\end{equation}
\end{propos}
\begin{proof}
The forms $\oq_1,\oq_2,\oq_3$ give a basis in the space of left invariant 1-forms on $SO(3)$. Let us introduce the left invariant fields $w^1,w^2,w^3$ such that the spin of $w^i$ is $\bi_i \in \bR^3$. The fields bracket $[w^1,w^2]$ is also left invariant and its spin due to \eqref{eq51} is $\bi_1{\times}\bi_2=\bi_3$, therefore
\begin{equation}\label{eq74}
  [w^1,w^2]=w^3.
\end{equation}
Analogously,
\begin{equation}\label{eq75}
  [w^2,w^3]=w^1, \qquad [w^3,w^1]=w^2.
\end{equation}
Now equations \eqref{eq73} follow from \eqref{eq74} and \eqref{eq75} since, obviously, the basis $\{\oq_1,\oq_2,\oq_3\}$ is dual to~$\{w^1,w^2,w^3\}$.
\end{proof}

Let us substitute \eqref{eq73} in \eqref{eq72} and restrict the form $\mg$ to the horizontal subspace $J_0$. The restriction is obtained just by replacing $\oq_i$ with $\oq_i^0$. We get
\begin{equation*}
\begin{array}{l}
\ds  \mg|{J_0}= \frac{1}{I_1\aq_1^2+I_2\aq_2^2+I_3\aq_3^2}\left[ I_1 \rmd \aq_1 \wedge \oq_1^0+ I_2 \rmd \aq_2 \wedge \oq_2^0+I_3 \rmd \aq_3 \wedge \oq_3^0 - \right.\\
\qquad   \left. - (I_1 \aq_1 \oq_2^0\wedge \oq_3^0+ I_2 \aq_2 \oq_3^0\wedge \oq_1^0+I_3 \aq_3 \oq_1^0\wedge \oq_2^0)\right].
\end{array}
\end{equation*}
Here we used the above mentioned property $I_1\aq_1\oq^0_1+I_2\aq_2\oq^0_2+I_3\aq_3\oq^0_3=0$. To find the form of gyroscopic forces of the reduced system $\tl{X}_k$, let us use the diffeomorphism \eqref{eq69}. We get
\begin{equation}\label{eq76}
\begin{array}{l}
\ds  k\,\tl{\mg}=k \frac{(I_2+I_3-I_1)I_1\aq_1^2+(I_3+I_1-I_2)I_2\aq_2^2+(I_1+I_2-I_3)I_3\aq_3^2}{(I_1\aq_1^2+I_2\aq_2^2+I_3\aq_3^2)^2} \times \\
   \qquad \times (\aq_1 \rmd \aq_2 \wedge \rmd \aq_3+\aq_2 \rmd \aq_3 \wedge \rmd \aq_1+\aq_3 \rmd \aq_1 \wedge \rmd \aq_2).
\end{array}
\end{equation}

The amended potential is found from \eqref{eq45}, \eqref{eq58}, \eqref{eq61}, and \eqref{eq63},
\begin{equation}\label{eq77}
\ds  \tl{V}_k(\aq_1,\aq_2,\aq_3)=  \tl{V}(\aq_1,\aq_2,\aq_3)+\frac{k^2}{2(I_1\aq_1^2+I_2\aq_2^2+I_3\aq_3^2)}.
\end{equation}

Finally, the reduced system in the dynamics of a rigid body is a mechanical system with gyroscopic forces
\begin{equation*}
  (S^2,\tl{m},\tl{V}_k, k\,\tl{\mg})
\end{equation*}
the elements of which are defined by \eqref{eq70}, \eqref{eq77}, and \eqref{eq76}.

Note that in the expression for the form of gyroscopic forces \eqref{eq76}, the multiplier
\begin{equation*}
  \aq_1 \rmd \aq_2 \wedge \rmd \aq_3+\aq_2 \rmd \aq_3 \wedge \rmd \aq_1+\aq_3 \rmd \aq_1 \wedge \rmd \aq_2
\end{equation*}
is the volume form of $S^2$ induced from $\bR^3$ and the coefficient in front of it, in the case $k\ne 0$, has constant sign on the sphere in virtue of the triangle inequalities for the inertia moments. Using Proposition~\ref{prop5} we get the following interesting property of the trajectories of the vertical direction vector on the Poison sphere: trajectories having the energy constant $h$ do not have inflection points in the metric $\tl{m}_h$. Standing on the outer side of the sphere we see that trajectories turn to the right of the corresponding geodesics when $k>0$ and to the left when $k<0$.

\end{document}